\documentclass[12pt]{amsart}

\usepackage{amssymb}
\usepackage{amscd}
\usepackage{statex}

\title[A theorem of Montanucci and Zini and its application]{A theorem of Montanucci and Zini for generalized Artin--Mumford curves and its application to Galois points}
\author{Satoru Fukasawa}

\thanks{The author was partially supported by JSPS KAKENHI Grant Number JP19K03438}
\address{Department of Mathematical Sciences, Faculty of Science, Yamagata University, Kojirakawa-machi 1-4-12, Yamagata 990-8560, Japan}
\email{s.fukasawa@sci.kj.yamagata-u.ac.jp}

\newtheorem{thm}{Theorem}

\newtheorem{lem}{Lemma}

\newtheorem{cor}{Corollary}
\newtheorem{fact}{Fact}

\theoremstyle{definition}
\newtheorem{rem}{Remark}

\begin{document}
\begin{abstract}
An elementary proof of a theorem of Montanucci and Zini on the automorphism group of generalized Aritn--Schreier--Mumford curves is presented, with the argument of Korchm\'{a}ros and Montanucci for Artin--Schreier--Mumford curves being improved. 
Although the characteristic of a ground field is assumed to be {\it odd} in the article of Montanucci and Zini, the proof in the present article is applicable to the case of characteristic two also. 
As an application of the theorem of Montanucci and Zini, the arrangement of Galois points  or Galois lines for the generalized Artin--Schreier--Mumford curve is determined. 
\end{abstract}
\maketitle

\section{Introduction} 
The Artin--Schreier--Mumford (ASM) curve over an algebraically closed field $k$ of characteristic $p>0$ is the smooth model of the plane curve defined by 
$$ (x^{p^e}-x)(y^{p^e}-y)=c, $$
where $e>0$, $p^e >2$ and $c \in k \setminus \{0\}$.   
This curve is important in the study of the automorphism groups of algebraic curves, since this is an ordinary curve and its automorphism group is large compared to its genus (see \cite{subrao}). 
The ASM curve is generalized as the smooth model $X$ of the curve $C$ defined by
$$ L_1 (x)\cdot L_2(y)+c=0, $$
where $c \in k \setminus \{0\}$ and $L_1$ and $L_2$ are linearlized polynomials of degree $p^e$, that is, 
$$L_i=\alpha_{i e} x^{p^e}+\alpha_{i e-1}x^{p^{e-1}}+\cdots+\alpha_{i 0} x$$ 
for some $\alpha_{i e}, \alpha_{i e-1}, \ldots, \alpha_{i 0} \in k$ with $\alpha_{i e}\alpha_{i 0} \ne 0$, for $i=1, 2$. 
We can assume that $\alpha_{1 e}=\alpha_{2 e}=1$ for a suitable system of coordinates. 
This curve was studied by the present author \cite{fukasawa1, fukasawa2} (for the case $L_1=L_2$), and by Montanucci and Zini \cite{montanucci-zini}. 
The curve $X$ is called a {\it generalized Artin--Mumford curve} in \cite{montanucci-zini}. The automorphism group ${\rm Aut}(X)$ of $X$ is completely determined by Montanucci and Zini \cite[Theorems 1.1 and 1.2]{montanucci-zini}, as follows. 

\begin{fact} \label{montanucci and zini}
Assume that $p >2$.  
Let $\mathbb{F}_{p^{k}}=\bigcap_{i >0, \alpha_{1 i} \ne 0}\mathbb{F}_{p^i} \cap \bigcap_{j>0, \alpha_{2 j} \ne 0}\mathbb{F}_{p^j}$. 
\begin{itemize} 
\item[(a)] If $L_1 = L_2$, then ${\rm Aut}(X) \cong \Sigma \rtimes D_{p^{k}-1}$, where $\Sigma$ is an elementary abelian $p$-group of order $p^{2e}$ and $D_{p^k-1}$ is the dihedral group of order $2(p^k-1)$. 
\item[(b)] If $L_1 \ne L_2$, then ${\rm Aut}(X) \cong \Sigma \rtimes \mathbb{F}_{p^{k}}^*$.  
\end{itemize}
\end{fact}

It is assumed that the characteristic is {\it odd} in \cite{montanucci-zini}. 
One key point to prove is \cite[Lemma 3.1 v) and Corollary 3.2]{montanucci-zini}, which asserts that a Sylow $p$-subgroup of ${\rm Aut}(X)$ is linear and acts on $\Omega_1 \cup \Omega_2$, where the set $\Omega_1$ (resp. $\Omega_2$) consists of all poles of $x$ (resp. of $y$). 
This assertion relies on a theorem of Nakajima \cite[Theorem 1]{nakajima} on relations between the $p$-rank and Sylow $p$-subgroups of the automorphism group of algebraic curves.   
Another key point is that the genus $(p^e-1)^2$ of $X$ is {\it even} if $p>2$, because Montanucci and Zini used some group-theoretic lemmas by a work of Giulietti and Korchm\'{a}ros \cite{giulietti-korchmaros} for algebraic curves of {\it even} genus. 

An alternative proof of Fact \ref{montanucci and zini} for the ASM curve was obtained by Korchm\'{a}ros and Montanucci \cite{korchmaros-montanucci}. 
It was proved that the linear system induced by some embedding into $\mathbb{P}^3$ is complete, and asserted that ${\rm Aut}(X)$ acts on $\Omega_1 \cup \Omega_2$, by using its completeness. 
We will prove the same things for generalized Artin--Schreier--Mumford curves in a different order. 
It was pointed out by Garcia \cite{garcia2} (see also \cite{boseck, garcia1}) that points of $\Omega_1 \cup \Omega_2$ are Weierstrass points (see Lemma \ref{weierstrass} for a more precise statement), and this implies that ${\rm Aut}(X)$ acts on $\Omega_1 \cup \Omega_2$. 
We reprove it. 
We also present an elementary proof of the completeness of the linear system for generalized ASM curves (Lemma \ref{complete}). 
With these two results combined, an inclusion ${\rm Aut}(X) \hookrightarrow PGL(4, k)$ is obtained (Corollary \ref{extendable in P^3}). 
More strongly: 

\begin{thm} \label{extendable in P^2}
There exists an injective homomorphism
$$ {\rm Aut}(X) \cong {\rm Bir}(C) \hookrightarrow PGL(3, k). $$ 
\end{thm}

This is very close to the theorem of Montanucci and Zini. 
Therefore: 

\begin{thm} \label{char 2}
The same assertion as Fact \ref{montanucci and zini} holds for the case where $p=2$. 
\end{thm} 

As an application of the theorem of Montanucci and Zini, the arrangement of Galois points or Galois lines for the generalized Artin--Schreier--Mumford curve is determined in Sections 3 and 4. 

\section{Proof of Theorems \ref{extendable in P^2} and \ref{char 2}}
The system of homogeneous coordinates on $\mathbb{P}^2$ is denoted by $(X:Y:Z)$ and the system of affine coordinates of $\mathbb{A}^2$ is denoted by $(x, y)$ with $x=X/Z$ and $y=Y/Z$. 

Let $q=p^e$. 
The set of all poles of $x$ (resp. of $y$) is denoted by $\Omega_1$ (resp. by $\Omega_2$), which coincides with the set of all zeros of $L_2(y)$ (resp. of $L_1(x)$). 
The sets $\Omega_1$ and $\Omega_2$ consist of $q$ points.  
The pole of $x$ (resp. of $y$) corresponding to $y=\beta$ (resp. $x=\alpha$) for $L_2(\beta)=0$ (resp. $L_1(\alpha)=0$) is denoted by $P_{\beta}$ (resp. by $Q_\alpha$). 
For the point $P_{\beta}$, $t=\frac{1}{x}$ is a local parameter. 
Let $P'=(1:0:0)$ and $Q'=(0:1:0) \in \mathbb{P}^2$. 
Then, ${\rm Sing}(C)=\{P', Q'\}$, and the point $P'$ (resp. the point $Q'$) is the image of $\Omega_1$ (resp. of $\Omega_2$) under the normalization.

Let $D$ be the divisor given by the line $\{Z=0\}$ of the plane model.  
It is known that the genus $g$ of $X$ is equal to $(q-1)^2$ (see \cite[Lemma 3.1]{montanucci-zini}, \cite[III. 7.10]{stichtenoth}). 
Therefore, the degree of the canonical divisor is $2g-2=2q(q-2)$. 
We consider the linear space $\mathcal{L}((q-2)D)$ associated with the divisor $(q-2)D$. 
The following two lemmas were proved by Boseck \cite{boseck} and Garcia \cite{garcia2} in a more general setting (see also \cite{garcia1}). 
We reprove them, for the convenience of the readers. 

\begin{lem} \label{canonical} 
A divisor $(q-2)D$ is a canonical divisor, and 
$$\mathcal{L}((q-2)D)=\langle x^i y^j \ | \ 0 \le i, j \le q-2 \rangle. $$
\end{lem}

\begin{proof}
Each element of the linear space $\mathcal{L}((q-2)D)$ is represented by a polynomial of $x$ and $y$, since each function $g \in \mathcal{L}((q-2)D)$ is regular on the affine open set $C \cap \{Z \ne 0\}$. 
Using the defining polynomial, since
$$ x^q y^q=\sum_{i \le q, j \le q, (i, j) \ne (q, q)} a_{ij} x^i y^j $$
in $k(X)$, it follows that any $g \in \mathcal{L}((q-2)D)$ is represented as a linear combination of monomials 
$$ x^i y^j \ \mbox{ with } i<q \ \mbox{ or } \ j <q. $$

Assume that there exists an element $g=\sum_{i=0}^m a_i(y) x^i \in \mathcal{L}((q-2)D)$ with $m \ge q$ and $a_m(y) \ne 0$. 
Then, $\deg a_m(y) <q$.
This implies that there exists a pole $P$ of $x$ such that ${\rm ord}_P a_m(y)=0$. 
Then, ${\rm ord}_P g=-m \le -q< -(q-2)$. 
This is a contradiction to $g \in \mathcal{L}((q-2)D)$.  
It follows that any $g \in \mathcal{L}((q-2)D)$ is represented as a linear combination of monomials 
$$ x^i y^j \ \mbox{ with } i<q \ \mbox{ and } \ j <q. $$

Let $g=\sum_{i=0}^m a_i(y) x^i$ with $m < q$. 
Since $\deg a_m(y) <q$, there exists a pole $P$ of $x$ such that ${\rm ord}_P a_m(y)=0$. 
Then, ${\rm ord}_P g=-m$. 
By the condition $g \in \mathcal{L}((q-2)D)$, $-m ={\rm ord}_P g \ge -(q-2)$. 
It follows that any $g \in \mathcal{L}((q-2)D)$ is represented as a linear combination of monomials 
$$ x^i y^j \ \mbox{ with } i \le q-2 \ \mbox{ and } \ j \le q-2. $$
These monomials are linearly independent, since $\deg C=2q$. 
Since $\deg (q-2)D=2g-2$ and $\dim \mathcal{L}((q-2)D) =g$, it follows from \cite[I.6.2]{stichtenoth} that the assertion follows.
\end{proof}

\begin{lem} \label{weierstrass} 
The set $\Omega_1 \cup \Omega_2$ coincides with $\{P \in X \ | \ q \in H(P)\}$, where $H(P)$ is the Weierstrass semigroup of $P$.  
In particular, all points of $\Omega_1 \cup \Omega_2$ are Weierstrass points. 
\end{lem}

\begin{proof}
We consider the embedding $\psi$ induced by the canonical linear system $|(q-2)D|$. 
Let $P \in \Omega_1$ and let $t=(1/x)$. 
Then, $t$ is a local parameter at $P$. 
Note that ${\rm ord}_P (y-\beta)=q$ for some $\beta \in k$. 
Considering the functions $t^k y^j \in \mathcal{L}((x^{q-2})+(q-2)D)$, it follows that the orders ${\rm ord}_P\psi^*H$ for hyperplanes $H \ni P$ are 
$$ 1, 2, \ldots, q-2, q, \ldots, $$
namely, $q-1+1$ is a {\it non-gap} (of pole numbers).  
On the other hand, ${\rm ord}_{(a, b)}(x-a)^{q-2}(y-b)=q-1$ for each point $(x, y)=(a, b)$ of $X \setminus (\Omega_1 \cup \Omega_2)$, since functions $x-a$ and $y-b$ are local parameters. 
\end{proof}

\begin{cor} \label{action}
The automorphism group ${\rm Aut}(X)$ preserves $\Omega_1 \cup \Omega_2$. 
\end{cor}

We consider the morphism 
$$ \varphi: X \rightarrow \mathbb{P}^3; \  (x:y:1:x y),  $$
similar to the case of the ASM curve (see \cite{korchmaros-montanucci}). 
For the point $P_{\beta} \in \Omega_1$ defined by $L_2(\beta)=0$,  $t=\frac{1}{x}$ is a local parameter at $P_{\beta}$. 
It follows that 
$$ \varphi=(x:y:1:x y)=(t x: t y: t: t x y)=(1:t y: t: y), $$
and $\varphi(P_{\beta})=(1:0:0:\beta)$. 
Therefore, $q$ points of $\varphi(\Omega_1)$ are contained in the line $Y=Z=0$ in $\mathbb{P}^3$ with a system of coordinates $(X:Y:Z:W)$. 
Similarly, $q$ points of $\varphi(\Omega_2)$ are contained in the line $X=Z=0$.

\begin{lem} \label{complete}
The morphism $\varphi$ is an embedding, and the linear system induced by $\varphi$ is complete. 
\end{lem}

\begin{proof} 
The former assertion is derived from the fact that the set $\varphi(\Omega_1 \cup \Omega_2)=\varphi(X)\cap \{Z=0\}$ consists of $2q$ points (this proof is similar to \cite{korchmaros-montanucci}). 
We consider the latter assertion. 
It follows that $1, x, y, x y \in \mathcal{L}(D)$. 
Each element of the linear space $\mathcal{L}(D)$ is represented by a polynomial of $x$ and $y$, since each function $g \in \mathcal{L}(D)$ is regular on the affine open set $C \cap \{Z \ne 0\}$. 
Using the defining polynomial, since
$$ x^q y^q=\sum_{i \le q, j \le q, (i, j) \ne (q, q)} a_{ij} x^i y^j $$
in $k(X)$, it follows that any $g \in \mathcal{L}(D)$ is represented as a linear combination of monomials 
$$ x^i y^j \ \mbox{ with } i<q \ \mbox{ or } \ j <q. $$

Assume that $g=\sum_{i=0}^m a_i(y) x^i$ with $m \ge q$ and $a_m(y) \ne 0$. 
Then, $\deg a_m(y) <q$.
This implies that there exists a pole $P$ of $x$ such that ${\rm ord}_P a_m(y)=0$. 
Then, ${\rm ord}_P g=-m \le -q< -1$. 
This is a contradiction to $g \in \mathcal{L}(D)$.  
It follows that any $g \in \mathcal{L}((D)$ is represented as a linear combination of monomials 
$$ x^i y^j \ \mbox{ with } i<q \ \mbox{ and } \ j <q. $$

Let $g=\sum_{i=0}^m a_i(y) x^i$ with $m < q$ and $a_m(y) \ne 0$. 
Since $\deg a_m(y) <q$, there exists a pole $P$ of $x$ such that ${\rm ord}_P a_m(y)=0$. 
Then, ${\rm ord}_P g=-m$. 
By the condition $g \in \mathcal{L}(D)$, $-m ={\rm ord}_P g \ge -1$. 
It follows that any $g \in \mathcal{L}(D)$ is represented as a linear combination of monomials $ 1, x, y, x y$.  
\end{proof}

\begin{cor} \label{extendable in P^3}
There exists an injective homomorphism 
$$ {\rm Aut}(X) \hookrightarrow PGL(4, k). $$
\end{cor} 

\begin{proof}  
By Corollary \ref{action}, $\sigma^*D=D$ for each $\sigma \in {\rm Aut}(X)$. 
By Lemma \ref{complete}, $\dim |D|=3$. 
The assertion follows. 
\end{proof}

Using Corollaries \ref{action} and \ref{extendable in P^3}, we prove Theorem \ref{extendable in P^2}.

\begin{proof}[Proof of Theorem \ref{extendable in P^2}]
Note that $\varphi(\Omega_1)$ and $\varphi(\Omega_2) \subset \mathbb{P}^3$ are contained in lines $Y=Z=0$ and $X=Z=0$ respectively. 
The point $R=(0:0:0:1)$ given by the intersection of such lines is fixed by each element of ${\rm Aut}(X)$. 
Then, ${\rm Aut}(X)$ acts on the linear subspace $\langle x, y, 1 \rangle \subset \mathcal{L}(D)$. 
\end{proof}

The image of the injective homomorphism described in Theorem \ref{extendable in P^2} is denoted by ${\rm Lin}(X)$. 
Let 
$$ \Sigma:=\{\sigma_{\alpha, \beta}: (x, y) \mapsto (x+\alpha, y+\beta) \ | \ L_1(\alpha)=0, L_2(\beta)=0\} \subset PGL(3, k), $$
$$ \Gamma:=\{\theta_\lambda: (x, y) \mapsto (\lambda x, \lambda^{-1} y) \ | \ \lambda \in \mathbb{F}_{p^k}^*\} \subset PGL(3, k), $$
and let $\tau \in PGL(3, k)$ be defined by $\tau(x, y)=(y, x)$. 
It follows that $\langle \Sigma, \Gamma \rangle \subset {\rm Lin}(X)$. 
If $L_1 =L_2$, then $\langle \Sigma, \Gamma, \tau  \rangle \subset {\rm Lin}(X)$. 

\begin{proof}[Proof of Fact \ref{montanucci and zini}] 
We prove that $\langle \Sigma, \Gamma \rangle= {\rm Lin}(X)$ if $L_1 \ne L_2$, and that $\langle \Sigma, \Gamma, \tau \rangle={\rm Lin}(X)$ if $L_1=L_2$. 
Let $P'=(1:0:0)$ and let $Q'=(0:1:0)$. 
Then, ${\rm Lin}(X)$ acts on ${\rm Sing}(C)=\{P', Q'\}$. 
Note that all tangent lines at $P'$ (resp. at $Q'$) are defined by $Y-\beta Z=0$ (resp. $X-\alpha Z=0$) for some $\beta \in k$ with $L_2(\beta)=0$ ($\alpha \in k$ with $L_1(\alpha)=0$). 
Since ${\rm Lin}(X)$ acts on the set of tangent lines at $P'$ or at $Q'$, 
it follows that there exists $\tau' \in {\rm Lin}(X)$ such that $\tau'(P')=Q'$ and $\tau'(Q')=P'$ if and only if $L_1=L_2$. 
Therefore, we prove that if $\sigma \in {\rm Lin}(X)$, $\sigma(P')=P'$ and $\sigma(Q')=Q'$, then $\sigma \in \langle \Sigma, \Gamma \rangle$.  

Assume that $\sigma \in {\rm Lin}(X)$, $\sigma(P')=P'$ and $\sigma(Q')=Q'$. 
Then, $\sigma$ is represented by a matrix
$$ A_{\sigma}=\left(
\begin{array}{ccc}
a & 0 & c \\
0 & b & d \\
0 & 0 & 1 
\end{array} \right)
$$
for some $a, b, c, d \in k$. 
Let $\beta \in k$ (resp. $\alpha \in k$) be a root of $L_2$ (resp. $L_1$). 
Then, the image of the tangent line $Y-\beta Z=0$ (resp. $X-\alpha Z=0$) under $\sigma$ is some tangent line $Y-\beta'Z=0$ (resp. $X-\alpha' Z=0$).
Then, an automorphism $\sigma_1:=\sigma_{\alpha-\alpha', \beta-\beta'} \circ \sigma$ fixes the tangent lines $Y-\beta Z=0$ and $X-\alpha Z=0$. 
Therefore,  this automorphism represented by 
$$ A_{\sigma_1}=\left(
\begin{array}{ccc}
a & 0 & \alpha (1-a) \\
0 & b & \beta (1-b) \\
0 & 0 & 1 
\end{array} \right). 
$$
Since
$$\sigma_1^* (L_1(x)\cdot L_2(y)+c)=\left(\sum_{i}a^{p^i}\alpha_{1 i}x^{p^i}+c_1\right)\left(\sum_{j}b^{p^j}\alpha_{2 j}y^{p^j}+c_2\right)+c=L_1(x)\cdot L_2(y)+c $$
up to a constant for some $c_1, c_2 \in k$, it follows that $c_1=c_2=0$, $ab=1$ and $a,b \in \mathbb{F}_{p^k}^*$. 
For an automorphism $\theta_{a^{-1}} \in \Gamma$, $\sigma_2:=\theta_{a^{-1}}\circ\sigma_{\alpha-\alpha', \beta-\beta'} \circ \sigma$ is represented by 
$$ A_{\sigma_2}=\left(
\begin{array}{ccc}
1 & 0 & \alpha(a^{-1}-1) \\
0 & 1 & \beta (a-1) \\
0 & 0 & 1 
\end{array} \right). 
$$
This implies that $\theta_{a^{-1}}\circ\sigma_{\alpha-\alpha', \beta-\beta'} \circ \sigma \in \Sigma$, namely, $\sigma \in \langle \Sigma, \Gamma \rangle$. 
\end{proof}

\section{Application to the arrangement of Galois points}
A point $R \in \mathbb{P}^2 \setminus C$ is called an outer Galois point for a plane curve $C \subset \mathbb{P}^2$ if the function field extension $k(C)/\pi_R^*k(\mathbb{P}^1)$ induced by the projection $\pi_R$ from $R$ is Galois (see \cite{miura-yoshihara, yoshihara}). 
Furthermore, an outer Galois point is said to be extendable if each element of the Galois group  is the restriction of some linear transformation of $\mathbb{P}^2$ (see \cite{fukasawa1}). 
The number of outer Galois points (resp. of extendable outer Galois points) is denoted by $\delta(C)$ (resp. by $\delta_0'(C)$). 

In this section, we consider outer Galois points for the plane model $C$ of the genelazied Artin--Schreier--Mumford curve $X$. 
According to Theorem \ref{extendable in P^2}, $\delta(C)=\delta_0'(C)$. 
It was proved by the present author that for the case where $L_1=L_2$, $\delta_0'(C) \ge p^k-1$ and the equality holds if $p=2$ (see \cite{fukasawa1, fukasawa2}). 
Therefore, it has been proved that $\delta(C)=\delta_0'(C)=p^k-1$ if $p=2$. 
The same holds for the case where $p>2$.  

\begin{thm} \label{outer Galois}
If $L_1=L_2$, then $\delta'(C)=\delta_0'(C)=p^k-1$. 
\end{thm}

\begin{proof}
Let $R \in \mathbb{P}^2 \setminus C$ be an outer Galois point. 
Note that the line $\overline{RP'}$ corresponds to the fiber of the projection $\pi_R$, where $P'=(1:0:0) \in {\rm Sing}(C)$.  
If $R \not \in \{Z=0\}$, then $\pi_R$ is ramified at each point of $\Omega_1$, since the Galois group $G_R$ acts on $\Omega_1 \cup \Omega_2$ (see \cite[III.7.1, III.7.2]{stichtenoth}). 
However, the directions of the tangent lines at $P'$ are different. 
This is a contradiction. 
Therefore, $R \in \{Z=0\}$. 

We can assume that $p>2$. 
Since $|G_R|=2q$, there exists an involution $\tau' \in G_R \subset {\rm Lin}(X)$. 
If $\tau'(P')=(P')$, then $\tau'$ fixes some point of $\Omega_1$. 
This is a contradiction to the transitivity of $G_R$ on fibers (see \cite[III.7.1]{stichtenoth}). 
Therefore, $\tau'(P')=Q'$ and $\tau'(Q')=P'$, where $Q'=(0:1:0) \in {\rm Sing}(C)$. 
Considering the elements of ${\rm Lin}(X)=\langle \Sigma, \Gamma, \tau \rangle$ described in the previous section, $\tau'$ is given by 
$$ (X: Y: Z) \mapsto (\lambda Y: \lambda^{-1}X: Z)$$
for some $\lambda \in \mathbb{F}_{p^{k}}^*$. 
Then, fixed points of $\tau'$ on $\{Z=0\}$ are $(\lambda:1:0)$ and $(-\lambda:1:0)$. 
Note that any element of $G_R$ fixes $R$, since $G_R$ preserves any line passing through $R$. 
Therefore, $R=(\lambda: 1: 0)$ or $(-\lambda:1:0)$.    
The claim follows. 
\end{proof}

\begin{rem}
Let $R_1, \ldots, R_{p^k-1}$ be all outer Galois points for $C$ and let $G_{R_1}, \ldots, G_{R_{p^k-1}}$ be their Galois groups. 
Then, ${\rm Aut}(X)=\langle G_{R_1}, \ldots, G_{R_{p^k-1}} \rangle$.  
\end{rem}

For the case where $L_1 \ne L_2$, the following holds. 

\begin{thm} \label{outer Galois 2}
If $L_1 \ne L_2$, then $\delta'(C)=0$. 
\end{thm} 

\begin{proof}
Let $R \in \mathbb{P}^2 \setminus C$ be an outer Galois point. 
Since the Galois group $G_{R}$ acts on $\Omega_1$ and $|G_R|=2q$, the projection $\pi_R$ is ramified at each point of $\Omega_1$ (see \cite[III.7.1, III.7.2]{stichtenoth}). 
However, the directions of the tangent lines at $P'$ are different. 
This is a contradiction. 
\end{proof}

\begin{rem}
The generalized Artin--Schreier--Mumford curve with $L_1 \ne L_2$ does not belong to families studied by the present author in \cite{fukasawa1, fukasawa2}, since some subgroup of the automorphism group of the families acts on the set defined by $Z=0$ transitively. 
\end{rem}

\section{Application to the arrangement of Galois lines}
A line $\ell \subset \mathbb{P}^3$ is called a Galois line for a space curve $X \subset \mathbb{P}^3$ if the function field extension $k(X)/\pi_\ell^*k(\mathbb{P}^1)$ induced by the projection $\pi_{\ell}$ from $\ell$ is Galois (see \cite{duyaguit-yoshihara, yoshihara2}). 
In this section, we consider Galois lines for a space model $\varphi(X) \subset \mathbb{P}^3$ of the generalized Aritn--Schreier--Mumford curve $X$, where 
$$ \varphi: X \rightarrow \mathbb{P}^3; \ (x: y: 1: x y).  $$
For the case where $\mathbb{F}_{p^k}=\mathbb{F}_{p^e}$, that is, $L_1=L_2=x^q+x$ (for a suitable system of coordinates), the arrangement of Galois lines was determined in \cite{fukasawa3}. 
We can assume that $k <e$. 

\begin{thm}
Assume that $L_1 \ne L_2$. 
Let $\ell \subset \mathbb{P}^3$ be a line. 
Then, $\ell$ is a Galois line for $\varphi(X)$ if and only if $\ell$ is defined by $a W-b X=a Y-b Z=0$ or $a W-b Y=a X-b Z=0$ for some $a, b \in k$ with $(a, b) \ne (0, 0)$. 
\end{thm}

\begin{proof} 
The proof of the if-part is similar to \cite{fukasawa3}, and is easily verified by a direct computation. 
Assume that $\ell \subset \mathbb{P}^3$ is a Galois line. 
If there exists a hyperplane $H \supset \ell \cup \Omega_1$, the claim follows, similar to \cite[Lemma 1 (b)]{fukasawa3}.  
Assume that $H \not\supset \Omega_1$ and $H \not\supset \Omega_2$ for each hyperplane $H \supset \ell$. 
Note that the Galois group $G_{\ell}$ acts on each fiber transitively (see \cite[III.7.1]{stichtenoth}). 
Since $\sigma$ acts on $\Omega_1$ and $\Omega_2$, it follows that a fiber containing a point of $P \in \Omega_1$ does not contain a point of $\Omega_2$. 
Therefore, each hyperplane $H \supset \ell$ with $H \cap (\Omega_1 \cup \Omega_2) \ne \emptyset$, the corresponding fiber contains $P$ only, namely, $G_\ell$ fixes each point of $\Omega_1 \cup \Omega_2$. 
This is a contradiction to the fact that the action of ${\rm Aut}(X)$ on $\Omega_1 \cup \Omega_2$ is faithful. 
\end{proof}

\begin{thm}
Assume that $L_1=L_2$ and $k<e$. 
Let $\ell \subset \mathbb{P}^3$ be a line. 
Then, $\ell$ is a Galois line for $\varphi(X)$ if and only if $\ell$ is one of the following: 
\begin{itemize}
\item[(a)] an $\mathbb{F}_{p^k}$-line contained in $\{Z=0\}$ and passing through $(0:0:0:1)$, or  
\item[(b)] the line defined by $W-a X=Y-a Z=0$ or $W-a Y=X-a Z=0$ for some $a\in k$. 
\end{itemize} 
\end{thm}

\begin{proof} 
The proof of the if-part is easily verified by a direct computation. 
Assume that $\ell$ is a Galois line. 
The proof for the case where $\ell \cap \varphi(X)=\emptyset$ and $\ell \not\ni (0:0:0:1)$ is similar to \cite{fukasawa3}. 
If $\ell \cap \varphi(X)=\emptyset $ and $\ell \ni (0:0:0:1)$, then such Galois lines correspond to Galois points in $\mathbb{P}^2$.  
According to Theorem \ref{outer Galois}, $\ell$ is an $\mathbb{F}_{p^k}$-line. 
If the degree of the projection from $\ell$ is $2q-1$, then $\ell$ is not a Galois line, by considering the orders $|G_{\ell}|$ and $|{\rm Aut}(X)|$. 

Assume that $\ell$ is a tangent line or $\ell \cap \varphi(X)$ consists of at least two points. 
If $\ell \subset \{Z=0\}$, then there exists $P \in \Omega_1$ and $Q \in \Omega_2$ such that $\ell=\overline{PQ}$, where $\overline{PQ} \subset \mathbb{P}^3$ is a line passing through $P$ and $Q$. 
Then, $\deg \pi_{\ell}=2q-2=2(q-1)$. 
Since $|{\rm Aut}(X)|=2q^2(p^k-1)$, $\ell$ must not become a Galois line. 
Assume that $\ell \cap \{Z=0\}=\{R\}$. 
If $R$ is contained in the line spanned by $\Omega_1$ or $\Omega_2$, then the claim follows, by \cite[Lemma 1 (b)]{fukasawa3}. 
Assume that $R$ is not contained in the lines spanned by $\Omega_1$ or $\Omega_2$. 
If there does not exist a pair of points $P, Q \in \Omega_1 \cup \Omega_2$ such that $P, Q \in H$, then $G_{\ell}$ fixes $\Omega_1 \cup \Omega_2$ pointwise. 
This is a contradiction. 
Therefore, there exist points $P \in \Omega_1$ and $Q \in \Omega_2$ such that $R \in \overline{PQ}$. 
Note that for any tangent hyperplane $H$ at $P$, ${\rm ord}_PH \ge q$. 
The same holds for $Q$. 
If $H$ contains the tangent lines at $P$ and at $Q$, then $\ell \cap \varphi(X)=\emptyset$.  
If $H$ is a tangent hyperplane at $P$, then $H$ is a tangent hyperplane at $Q$ also, by \cite[III.7.2]{stichtenoth}. 
We can assume that $(H \setminus \ell) \cap \varphi(X)=\{P, Q\}$. 
It follows that $\deg \pi_\ell=2$. 
According to Lemma \ref{canonical}, $X$ is not hyperelliptic. 
This is a contradiction. 
\end{proof}

\

\begin{center}
{\bf Acknowledgments} 
\end{center}
The author is grateful to Doctor Kazuki Higashine for helpful discussions.


\begin{thebibliography}{100} 
\bibitem{boseck} H. Boseck, Zur Theorie der Weierstra\ss punkte, Math. Nachr. {\bf 19} (1958), 29--63. 
\bibitem{duyaguit-yoshihara} C. Duyaguit and H. Yoshihara, Galois lines for normal elliptic space curves, Algebra Colloq. {\bf 12} (2005), 205--212.
\bibitem{fukasawa1} S. Fukasawa, Galois points for a plane curve in characteristic two, J. Pure Appl. Algebra {\bf 218} (2014), 343--353. 
\bibitem{fukasawa2} S. Fukasawa, A family of plane curves with two or more Galois points in positive characteristic, Contemporary Developments in Finite Fields and Applications, 62--73, World Sci. Publ., 2016. 
\bibitem{fukasawa3} S. Fukasawa, Galois lines for the Artin--Schreier--Mumford curve, preprint, arXiv:2005.10073. 
\bibitem{garcia1} A. Garcia, On Weierstrass points on Artin--Schreier extensions of $k(x)$, Math. Nachr. {\bf 144} (1989), 233--239. 
\bibitem{garcia2} A. Garcia, On Weierstrass points on certain elementary abelian extensions of $k(x)$, Comm. Algebra {\bf 17} (1989), 3025--3032.  
\bibitem{giulietti-korchmaros} M. Giulietti and G. Korchm\'{a}ros, Algebraic curves with many automorphisms, Adv. Math. {\bf 349} (2019), 162--211.
\bibitem{korchmaros-montanucci} G. Korchm\'{a}ros and M. Montanucci, The geometry of the Artin--Schreier--Mumford curves over an algebraically closed field, Acta Sci. Math. (Szeged) {\bf 83} (2017), 673--681.  
\bibitem{miura-yoshihara} K. Miura and H. Yoshihara, Field theory for function fields of plane quartic curves, J. Algebra {\bf 226} (2000), 283--294.
\bibitem{montanucci-zini} M. Montanucci and G. Zini, Generalized Artin--Mumford curves over finite fields, J. Algebra {\bf 485} (2017), 310--331. 
\bibitem{nakajima} S. Nakajima, $p$-ranks and automorphism groups of algebraic curves, Trans. Amer. Math. Soc. {\bf 303} (1987), 595--607. 
\bibitem{stichtenoth} H. Stichtenoth, Algebraic Function Fields and Codes, Universitext, Springer-Verlag, Berlin (1993).
\bibitem{subrao} D. Subrao, The $p$-rank of Artin--Schreier curves, manuscripta math. {\bf 16} (1975), 169--193. 
\bibitem{yoshihara} H. Yoshihara, Function field theory of plane curves by dual curves, J. Algebra {\bf 239} (2001), 340--355.
\bibitem{yoshihara2} H. Yoshihara, Galois lines for space curves, Algebra Colloq. {\bf 13} (2006), 455--469.  
\end{thebibliography}
\end{document}